\documentclass[a4paper,12pt]{article}
\usepackage{a4}
\usepackage[english]{babel}
\usepackage[T1]{fontenc}
\usepackage[all]{xy}
\usepackage{amssymb}
\usepackage{amsmath}
\usepackage{amsthm}
\newtheorem{thm}{Theorem}[section]
\newtheorem{prop}[thm]{Proposition}

\newtheorem{coro}[thm]{Corollary}
\newtheorem{defn}[thm]{Definition}

\newtheorem{rmq}[thm]{Remark} 
\numberwithin{equation}{section}  

\newcommand{\GL}{\operatorname{GL}}

\newcommand{\Fr}{\operatorname{Framed}}

\title{ \textbf{Rational homotopy type of the space of immersions of a manifold in an euclidian space}}
\date{}

\author{ Abdoul Kader Yacouba Barma}
\begin{document}
\maketitle
\begin{abstract}
Let $M$ be a simply-connected  $m$ dimensional manifold of finite type and $k$ a positif integer. In this paper we show that the rational Betti numbers  
of each component of the space of immersions of $M$ in $\mathbb{R}^{m+k}$, have polynomial growth. As consequence, we deduce
that, if $M$ is a manifold with Euler characteristic $\chi\left(M\right)\leq -2$, the Betti numbers of smooth embeddings, $Emb\left(M ,\mathbb{R}^{m+k}\right)$, have exponential growth if $k\geq  m+1$.\\
 The main tool of this work
is the construction of an explicit model of the space of immersions.
\end{abstract}
\section{Introduction}
Throughout the paper we fix a smooth differential manifold, $M$, of dimension $m$.  We assume that $M$ is simply-connected and of finite type.
An immersion of $M$ in $\mathbb{R}^d$ is a smooth map $f$ from $M $ to $\mathbb{R}^d$ such that  the differential of $f$ at each $x\in M$ is of rank $m$.
The space of immersions  of $M$ in $\mathbb{R}^d$ is the set
\[
Imm\left(M,\mathbb{R}^d\right)=\{ f\colon M\looparrowright{\mathbb{R}^d} \mid f \text{ is an immersion}\},
\]  
equipped with the weak $C^1-$topology.\\
In this paper we determine the rational homotopy type of the path components of $Imm\left(M,\mathbb{R}^{m+k}\right)$. When $k$ is odd and $\neq 0$, that rational homotopy  is easy to describe: 
\begin{thm} [Theorem \ref{casod}]

Let $M$ be a simply connected $m$-manifold of finite type. And let $k=2s+1$ with $s\neq 0$. Suppose that $p_i\left(\tau_M\right)=0, \forall i\geq s+1$,then each connected component  of $Imm\left(M,\mathbb{R}^{m+k}\right)$ has the rational homotopy type of a product of Eilenberg-Maclane spaces, which depends only on the Betti numbers of $M$ and on $k$.
\end{thm}

When the codimension, $k$, is even and $\geq 2$, the rational homotopy type of the path components of $Imm\left(M,\mathbb{R}^{m+k}\right)$ is more complicated to describe. In that case we obtain the following results.
\begin{thm}[Theorem \ref{caspair}]

Let $M$ be a simply connected $m$-manifold of finite type and $k$ an even integer. Then:
\begin{itemize}
\item
If $k\geq m+1$,  $Imm\left(M,\mathbb{R}^{m+k}\right)$ is connected and its rational homotopy type is of the form
 \[
 Imm\left(M, \mathbb{R}^{m+k}\right)\simeq_{\mathbb{Q}}Map\left(M ,S^k\right)\times K,
 \]
 where
 $Map\left(M ,S^k\right)$ is the mapping space of maps from $M$ into $S^k$ and $K$ is a product of Eilenberg-Maclane spaces which depend only on the Betti numbers of $M$.
 \item
 If $2\leq k\leq m$,  Suppose that $p_i\left(\tau_M\right)=0, \forall i\geq s$ then, the connected component of $f$ in $Imm\left(M,\mathbb{R}^{m+k}\right)$ has the rational homotopy type of a product
 \[
  Imm\left(M, \mathbb{R}^{m+k},f\right)\simeq_{\mathbb{Q}}Map\left(M ,S^k,\tilde{f}\right)\times K,
 \]
where
 $ Imm\left(M, \mathbb{R}^{m+k},f\right)$ is the component of  $Imm\left(M, \mathbb{R}^{m+k}\right)$ which contains $f$,\\
$ K$ is the product of Eilenberg-Maclane spaces which depends only on the Betti numbers of $M$,\\
 $Map\left(M ,S^k,\tilde{f}\right)$  is the path component of the  mapping space of $M$ in $S^k$ containing the map $\tilde{f}$ induced by the immersion $f$.
 \end{itemize}
\end{thm}
As the special case of this theorem we have.
\begin{thm}[Corollary~\ref{idem}]
If $H^k\left(M,\mathbb{Q}\right)=0$, then all the path connected of $Imm\left(M,\mathbb{R}^{m+k}\right)$ have the same rational homotopy type, and we can take $\tilde{f}$ to be the constant map.
\end{thm}
As a consequence of Theorem~\ref{casod} and Theorem~\ref{caspair} we have the following corollary 
\begin{thm} [Corollary~\ref{betti}]

If $M$ is simply connected manifold of dimension m and k an integer such that $k\geq 2$. Then the rational Betti numbers of each component of $Imm\left(M, \mathbb{R}^{m+k}\right)$ have polynomial growth.
\end{thm}
From this result we deduce the following.
\begin{thm}[Remark~\ref{alp}]
Let $M$ be a simply connected m-manifold of finite type with Euler characteristic $\chi\left(M\right)\leq -2$, and $k\neq 0$ an integer. Then if $k\geq m+1$, the Betti numbers of the space of smooth embedding of $M$ into $\mathbb{R}^{m+k}$, $Emb\left(M,\mathbb{R}^{m+k}\right)$, have exponential growth. 
\end{thm}
In \cite{hirsh}, M.W Hirsh proves that $Imm\left(M,\mathbb{R}^{m+k}\right)$ has the same homotopy type of the space of sections of a bundle $\Fr_m\left(\tau_M\right)$(see \ref{stiefelbundle} for the definition of $\Fr_m\left(\tau_M\right))$ whose fiber is the Stiefel manifold $V_m\left(\mathbb{R}^{m+k}\right)$ of m-frames in $\mathbb{R}^{m+k}$. That is,
\[
Imm\left(M,\mathbb{R}^{m+k}\right)\simeq \Gamma\left(\Fr_m\left(\tau_M\right)\right).
\] 

The main ingredients of the proof of Theorem~\ref{casod} and Theorem~\ref{caspair} is  the rational model of $\Fr_m\left(\xi\right)$(see Propositon~\ref{prop:principal} for its construction) from which we deduce the following result.
\begin{thm}[Proposition~\ref{trivial}]
Let $\xi$ be a vector bundle of rank $m$ and $p_i\left(\xi\right), 1\leq i\leq \frac{m}{4}$, its Pontryagin classes. If $p_i\left(\tau_M\right)=0 \text{ }\forall \text{ }i\geq \frac{k}{2}$, then $\Fr_m\left(\xi\right)$ is rationally trivial.
 \end{thm}

\textbf{Outline of the paper}
\begin{enumerate}
\item[-] In Section 2 we will construct the framed bundle associated to a vector bundle. We will start with the construction of the principal bundle associated to a vector bundle, afterward we will construct a framed bundle associated to a vector bundle and at the end we will prove that the framed framed bundle is a pullback of  Borel bundle.
\item[-]  In Section 3  we will construct the model of framed bundle. We will start with a review of the construction of the model of bundle, later we recall the model of Borel bundle in general case. We will prove that under certains conditions this bundle is rational trivial.
\item[-]  In Section 4 we will prove Theorem~\ref{casod}, and Theorem~\ref{caspair} and we will prove the polynomial growth of the space of embeddings. We will start this section with a quick review of the Small Hirsch theorem.
\end{enumerate}

\textbf{Acknowlegment}.\\ The present work is a part of my thesis. I would like to thank my advisor, Pascal Lambrechts, for making it possible throughout his advices
and encouragements. I also thank Yves Felix for suggesting that Corollary~\ref{idem} is  true without a codimension condition.
\section{Framed bundle associated to a vector bundle}

In this section we  study the framed bundle associated to a vector bundle. First we recall its construction and give some of its property which allow us to identify it with a pullback  of  certain Borel bundle. 
\subsection{Generality}
\label{stiefelbundle}
We start with some generality.
\subsubsection{$GL\left(m\right)$ prinicipal bundle}
 Let $\xi=\left(E,\pi,B\right)$ be a real  vector bundle of rank $m$. For $b\in B$, we denote by $E_b\xi$ the fiber of $\xi$ over $b $:
\[
E_b\xi:=\pi^{-1}\left(b\right).
\]
We associate to $\xi$ the $GL\left(m\right)$-principal bundle $P\xi=\left( P_m\xi,p,B\right)$. Where $P_m\xi$  is the set of pairs $\left(b,\left(v_1\cdots,v_m\right)\right)$ with $b\in B$ and $\left(v_1,\cdots,v_m\right)$ is an $m-$frames in $E_b$. $P_m\xi$ is topologize as a subspace of  the Whitney sum $E\oplus\cdots\oplus E=mE$. $p$ denote the  projection map from $P_m\xi$ to $B$ , which send the frame $\left(b,\left(v_1\cdots,v_m\right)\right)$  on $b$. Note that $GL\left(m\right)$ acts freely on the right of $P_m\xi$. The action is defined as follow. Let $g=\{g_{ij}\}\in GL\left(m\right)$ and $\left(b,\left(v_1\cdots,v_m\right)\right)\in P_m\xi$.
\[
\left(b,\left(v_1\cdots,v_m\right)\right)g=\left(b,\left(w_1,\cdots,w_m\right)\right)
\]
where
\[
w_j=\sum_{i}v_ig_{ij}.
\]
This make $P\xi=\left( P_m\xi,p,B\right)$ into a principal bundle over $B$.

\subsubsection{Associated framed bundle}
Let $W$ be a vector space of dimension $m+k$. We denote by $V_mW$ the space of $m$ frames in $W$. $V_mW$ is a right $GL\left(m\right)$ space. 
By consequence $GL\left(m\right)$ acts on  $P_m\xi\times V_mW$. Set by $P_m\xi\times_{GL\left(m\right)} V_mW$  the orbit of $P_m\xi\times V_mW$ under this action, and
\[
q: P_m\xi\times V_mW\longrightarrow P_m\xi\times_{GL\left(m\right)} V_mW
\]
the corresponding projection. $q$ determine a map 

\[
\rho:P_m\xi\times_{GL\left(m\right)} V_mW\longrightarrow B
\]
via the commutative diagram

 \[
 \xymatrix{
  P_m\xi\times V_mW   \ar[r]^q \ar[d]_{pr_1}  &P_m\xi\times_{GL\left(m\right)}V_mW \ar[d]^\rho \\
    P_m\xi \ar[r]^p& B
  }
  \]
  
\begin{prop}\cite[Proposition~1, p~198]{halperin2}
There is an unique topology on $P_m\xi\times_{GL\left(m\right)}V_mW$ such that  $\Fr_m\left(\xi\right)=\left(P_m\xi\times_{GL\left(m\right)}V_mW,\rho,B\right)$ is a locally trivial bundle with fiber $V_mW$.
\end{prop}  

\begin{defn}
$\Fr_m\left(\xi\right)$ is called framed bundle with fiber $V_mW$  associated to $\xi$.
\end{defn}
Denote by  $P_0=\left(EGL\left(m\right),p_0,BGL\left(m\right)\right)$ the universal $GL\left(m\right)$ principal bundle, where $EGL\left(m\right)$ is a contractile space and $BGL\left(m\right)=\frac{EGL\left(m\right)}{GL\left(m\right)}$.  By the theorem of classification of principal bundle there exist  a map
\[
f : B\longrightarrow BGL\left(m\right)
\]
 such that 
$P_m\xi$ is the pullback of $P_0$ via $f$. $f$ is called the classifying map. Consider the Borel bundle $ \left(EGL\left(m\right)\times_{GL\left(m\right)}V_mW,\tilde{p}_0,BGL\left(m\right)\right)$ with fiber $V_mW$  associated to $P_0=\left(EGL\left(m\right),p_0,BGL\left(m\right)\right)$ . Then we have the following result.

\begin{prop}
$\Fr_m\left(\xi\right)$ is the pullback of $ \left(EGL\left(m\right)\times_{GL\left(m\right)}V_mW,\tilde{p}_0,BGL\left(m\right)\right)$ via $f$.
\end{prop}
   \begin{proof}
 Since $P\xi=\left( P_m\xi,p,B\right)$ is  the pullback of $P_0=\left(EGL\left(m\right),p_0,BGL\left(m\right)\right)$ via $g$. This means that, if  $T$ is a space and $p_1: T\longrightarrow EGL(m), p_2: T\longrightarrow B$ are maps such that 
 \[
 p_0\circ p_1=g\circ p_2
 \]
 then there is a unique map:
 
 \[
  \Phi: T\longrightarrow P_m\xi
  \]
  such that the following diagram is commutative  
   \[
 \xymatrix{
 T \ar@/^/[rrd] ^{p_1}\ar@/_/[rdd]_{p_2} \ar@{.>}[rd]^{\Phi} \\
    &P_{m}\xi \ar[r]^F \ar[d]^p  & EGL\left(m\right) \ar[d]^{p_0 }\\
    &B\ar[r]_f  & BGL\left(m\right)
  }
  \]
  Therefore, if $T'$ is a space, 
  \[
  p_1':T\longrightarrow EGL\left(m\right)\times_{GL\left(m\right)}V_mW
  \] 
  and
  \[
  p_2':T\longrightarrow B
  \]
  two  maps such that 
  \[
  \tilde{p}_0\circ p_1'=g\circ p_2' 
  \]
   there is a unique map
   \[
\Phi': T'\longrightarrow P_m\xi\times_{GL\left(m\right)}V_mW
\]
 such that the following diagram is commutative 
  \[
 \xymatrix{
 T' \ar@/^/[rrd]^{p_1'} \ar@/_/[rdd]_{p_2'} \ar@{.>}[rd]^{\Phi'} \\
    & P_m\xi\times_{\GL\left(m\right)}V_mW\ar[r]^F'\ar[d]^{\rho}  & EGL\left(m\right)\times_{\GL\left(m\right)}V_mW\ar[d]^{\tilde{p}_0} \\
    &B\ar[r]_f  & BGL\left(m\right)
  }
  \]
  with
  \[
  F'\left(\left[X,v\right]\right)=\left[F\left(x,\lambda\right),v\right],
  \]
  \[
  \rho\left(\left[X,v\right]\right)=p\left(X\right)
    \]
     \[
  \tilde{ p}_0\left(\left[e,v\right]\right)=p\left(e\right).
  \]
  
  This proves that $\Fr_m\left(\xi\right)$ is the pullback of $ \left(EGL\left(m\right)\times_{GL\left(m\right)}V_mW,\tilde{p}_0,BGL\left(m\right)\right)$ via $f$.
  \end{proof} 

\section{Model of associated bundle }
In this section we construct the rational model of de framed bundle associated to a vector bundle construct in Section 2. We start with some notions of rational homotopy theory. Afer we give model of the Borel bundle constructed by  K. Matsuo in \cite{kent}, and give this construction in a particular case. And finally we will use this model of Borel bundle to construct the model of a framed bundle.
\subsection{Basic notions}
In this section we recall some classics notions of rational homotopy theory. 
\subsubsection{Rational homotopy theory}
In all this paper we wil use the standard tools of rational homotopy theory, following the notation and terminologie of \cite{rht}. Recall that $A_{pl}$ is the Sullivan- de Rham contravariant functor and that for a simply connected space of finite type , $A_{pl}\left(X\right)$ is a CGDA(commutative differential graded algebra). Any CGDA weakly equivalent to $A_{pl}\left(X\right)$ is called a CGDA model of $X$ and it completely encodes the rational homotopy of $X$.
\subsubsection{Model of bundle}
Let $\xi=\left(E,p,B\right)$ be a bundle of fiber $F$. Assume that $F,B$ are simply connected CW complexes of finite type. If $\left(A,d_A\right)$ is a CGDA model of $B$. In \cite[Theorem~15.3]{rht} it's proved that we can form a commutative diagramme 
\[
\xymatrix@!C{
    A_{pl}\left(B\right) \ar[r]^{A_{pl}\left(p\right)} & A_{pl}\left(E\right) \ar[r]& A_{pl}\left(F\right)  \\
    \left(A,d_A\right)  \ar[u]^{m}\ar@{^{(}->}[r]& \left(A\otimes \Lambda V,d\right)\ar[r]\ar[u]{\phi}& \left(\Lambda V,\bar{d} \right)\ar[u]^{\bar{\phi}}}
  \]
  where $\phi$, $\bar{\phi}$ and  $m$ are quasi isomorphism. 
  \[
  \begin{cases}
  d_{\mid{A}}=d_A\\
  dV\in A\otimes \Lambda V.
  \end{cases}
  \]
 \begin{defn} The inclusion
 \[
 \left(A,d_A\right)\hookrightarrow\left(A\otimes \Lambda V,d\right)
 \]
 is called the model of $\xi=\left(E,p,B\right)$. We will said that its rationally trivial if
 \[
 \left(A\otimes \Lambda V,d\right)\simeq \left(A,d_A\right)\otimes\left(\Lambda V, d\right).
  \]
  
 \end{defn}
 \begin{rmq}
 If $\xi_{\mathbb{Q}}=\left(E_{\mathbb{Q}},p_{\mathbb{Q}},B_{\mathbb{Q}}\right)$ is the rationalization of $\xi$.  If $\xi$ is rationally trivial, then $\xi_{\mathbb{Q}}$
 is trivial.
  \end{rmq}
 And we have the following result proved in \cite{lupton}
 \begin{prop}
 \label{lupton}
 Let $\xi=\left(E,p,B\right)$ be a bundle of fibre $F$ with $F$ and $B$ simply connected. Set $\xi_{\mathbb{Q}}=\left(E_{\mathbb{Q}},p_{\mathbb{Q}},B_{\mathbb{Q}}\right)$ the rationalization of $\xi$ and  $\xi_{\left(\mathbb{Q}\right)}=\left(E_{\left(\mathbb{Q}\right)},p_{\left(\mathbb{Q}\right)},B_{\left(\mathbb{Q}\right)}\right)$ its fiberwise rationalization. If
 $\xi_{\mathbb{Q}}$ is trivial then $\xi_{\left(\mathbb{Q}\right)}$ is trivial.
   \end{prop}
   \begin{proof}
   It comes from the unicity of  $p_{\left(\mathbb{Q}\right)}: E_{\left(\mathbb{Q}\right)}\longrightarrow B$ up to homotopy pvoved by LIerena \cite[Proposition~6.1]{irene}, and the fact that it's the pullback of $p_{\mathbb{Q}}: E_{\mathbb{Q}}\longrightarrow B_{\mathbb{Q}}$ along $B\longrightarrow B_{\mathbb{Q}}$
    \end{proof}

\subsection{Model of associated bundle with fiber a homogenous space}
\label{modelfb}
 $G$ is a simply connected Lie group.\\
Let $\mathcal{P}=\left(P,p,B\right)$ be a $G$-principal bundle. It  is pull back from the universal bundle $\left(EG,p,BG\right)$ via a classifying map $\phi: B\longrightarrow BG$, and then the associated bundle $\left(P\times_GF,p,B\right)$ is the pullback of $\left(EG\times_GF,p,BG\right)$ via $\phi$.\\
Suppose that $F$ is a homogenous space, that is ,$F=\frac{H}{K}$, where $H$ is a compact Lie group simply connected and $K$ a simply connected closed subgroup of $H$. If there is  a morphism of Lie group $\mu: G\longrightarrow H$ then, $\frac{H}{K}$ admets the action of $G$ defined by $ghK=\left(\mu\left(g\right)h\right)K$. Since $G,H, K$ are simply connected $BG,BH$, and $BK$ are simply connected and by the result of Borel \cite[Theorem 19.1]{borel} their  rational cohomology are finitely generated polynomial algebras $\Lambda V_G,\Lambda V_H$ and $\Lambda V_K$. In particular thier minimals models are given by: $\left(\Lambda V_G,0\right),\left(\Lambda V_H,0\right)$ and $\left(\Lambda V_K,0\right)$. Consider $B\mu: BG\longrightarrow BH$ and $B\nu :BK\longrightarrow BH$ the morphisms induced by $\mu$ and the inclusion of $\nu$ of $K$ in $H$ . K. Matsuo prove in \cite{kent} that the model of :
\[
EG\times_G^\mu \frac{H}{K}\longrightarrow BG
\]
is of the form
\[
\left(\Lambda V_G,0\right)\hookrightarrow \left(\Lambda V_G\otimes \Lambda V_K\otimes \Lambda sV_H,d\right)
\]
where $(sV_H)^i= V_H^{i+1}$ and
\[
\begin{cases}
d_{\mid V_G}=d_{\mid V_K}=0\\
dsv=(B\mu)^{\ast}\left(v\right)-(B\nu)^{\ast}\left(v\right) \text{ for } sv\in sV_H.
\end{cases}
\]
Let $\phi^\ast :\left(\Lambda V_G,0\right)\longrightarrow\left(A,d_A\right)$ be the Sullivan representative of $\phi$. Then combining and \cite[Theorem~270]{amt08} we obtain the following result.
\begin{prop}
\label{model1}
The model of 
\[
P\times_G^{\mu}\frac{H}{K}\longrightarrow B
\]
is of the form
\[
\left(A,d_A\right)\hookrightarrow \left(A\otimes \Lambda V_K\otimes \Lambda sV_H,D\right)
\]
where
\[
\begin{cases}
D_{\mid A}=d_A\\
D_{\mid V_K}=0\\
Dsv=\phi^\ast\left((B\mu)^{\ast}\left(v\right)\right)-(B\nu)^\ast\left(v\right).
\end{cases}
\]
\end{prop}

  \subsection{Model of framed associated bundle}
  \label{rationallytrivial}
  In this section we will construct  a model of framed bundle associated to the vector bundle $\xi=\left(E,p,B\right)$. We suppose that $B$ is simply connected. In this case, the $GL\left(m\right)$ structure on its associated principal bundle $P\xi$ can be restreinte to $SO\left(m\right)$ . We suppose also that $W=\mathbb{R}^{m+k}$ with $k\geq 1$, in this case we have an $SO\left(m\right)$ equivariant homeomorphism between $\frac{SO\left(m+k\right)}{SO\left(k\right)}$ and $V_m\left(\mathbb{R}^{m+k}\right)$. This homeomorphism induce an isomorphism of bundle over $B$ between $\Fr_m\left(\xi\right)=\left(P_m\xi\times_{SO\left(m\right)}V_m\left(\mathbb{R}^{m+k}\right),\rho,B\right)$  and $\left(P_m\xi\times_{SO\left(m\right)}\frac{SO\left(m+k\right)}{SO\left(k\right)},\tilde{\rho},B\right)$.\\
  We will use the following notation for certains cohomology classes of $BSO\left(m\right),BSO\left(k\right)$ and $BSO\left(m+k\right)$:
  
  \begin{itemize}
  \item
  $p_i\in H^{4i}\left(BSO\left(m\right),\mathbb{Q}\right); i:1\cdots,[\frac{m-1}{2}]$ is the universal Pontryagin classes,\\
  $e_m\in H^{m}\left(BSO\left(m\right),\mathbb{Q}\right)$, for $m$ even is Euler classes,
  \item
   $b_i\in H^{4i}\left(BSO\left(k\right),\mathbb{Q}\right); i:1\cdots,[\frac{k-1}{2}]$ is the universal Pontryagin classes,\\
  $e_k\in H^{k}\left(BSO\left(k\right),\mathbb{Q}\right)$, for $k$ even is Euler classes,
  \item
   $c_i\in H^{4i}\left(BSO\left(m+k\right),\mathbb{Q}\right); i:1\cdots,[\frac{k-1}{2}]$ is the universal Pontryagin classes,\\
  $e_{m+k}\in H^{m+k}\left(BSO\left(m+k\right),\mathbb{Q}\right)$, for $m+k$ even is Euler classes.

  \end{itemize}
 \begin{prop}
   \label{prop:principal}
Let 
$\xi := \left(E,\pi,B\right)$ be a vectoriel bundle with fibre $\mathbb{R}^m$
and
\[
 \left( A,d_A\right)\overset{\simeq}\longrightarrow{A_{pl}\left(B\right)}
\]
a CGDA model of B. If $ p_i\left(\xi\right) \in {A \cap {kerd_A}}$ are the  Pontrjagin classes  of $\xi$, then the model of
\[
\tilde{\rho}:P_m\xi\times_{SO\left(m\right)}\frac{SO\left(m+k\right)}{SO\left(k\right)}\longrightarrow{B}
\]
is of the form :

 \begin{itemize}
 \item{if $m=2l+1$,$k=2s+1$}
 \[
 \left( A,d_A\right)\longrightarrow \left({A\otimes\Lambda{\left(x_{s+1},\cdots,x_{l+s},\bar{e}_{m+k-1}\right)}}, D\right)
 \]
\[ 
\begin{cases} 
 Da=d_Aa \text{ if } a\in A\\
 Dx_i =p_i\left(\xi\right) \text{ if } i\leq \frac{m}{4},
 Dx_i=0 \text{ if } i>\frac{m}{4}\\
 D\bar{e}_{m+k-1}=0, 
 \end{cases}
 \]
 \item{if $m=2l$, $k=2s+1$}
 \[
 \left( A,d_A\right)\longrightarrow \left({A\otimes\Lambda{\left(x_{s+1},\cdots,x_{l+s}\right)}}, D\right)
 \]
\[ 
\begin{cases} 
 Da=d_Aa \text{ if } a\in A\\
 Dx_i =p_i\left(\xi\right) \text{ if } i\leq \frac{m}{4},
 Dx_i=0 \text{ if } i>\frac{m}{4}\\
 \end{cases}
 \]
\item{if $m=2l+1$,$k=2s$}
 \[
 \left( A,d_A\right)\longrightarrow \left({A\otimes\Lambda{\left(x_s,\cdots,x_{l+s},e_k\right)}}, D\right)
 \]
\[ 
\begin{cases} 
 Da=d_Aa \text{ if } a\in A\\
 Dx_s=e_k^2+p_s\left(\xi\right)\\
 Dx_i =p_i\left(\xi\right) \text{ if } i\leq m,
 Dx_i=0 \text{ if } i>m\\
 De_k=0
 \end{cases}
 \]
 \item{if $m=2l$, $k=2s$}
  \[
 \left( A,d_A\right)\longrightarrow \left({A\otimes\Lambda{\left(x_s,\cdots,x_{l+s-1},\bar{e}_{m+k-1},e_k\right)}}, D\right)
\]
\[ 
\begin{cases} 
 Da=d_Aa \text{ if } a\in A\\
 Dx_s=e_k^2+p_s\left(\xi\right)\\
 Dx_i =p_i\left(\xi\right) \text{ if } s< i\leq m,
 Dx_i=0 \text{ if } i>m\\
 D\bar{e}_{m+k-1}=De_k=0.
 \end{cases}
 \]
 
 \end{itemize}
 \end{prop}
 \begin{proof}
 For the proof we consider the case $m=2l+1$ and $k=2s+1$. The proofs of the others cases are similar.\\
 By \ref{model1}, the model of 
\[
\tilde{\rho}:P_m\xi\times_{SO\left(m\right)}V_m\left(\mathbb{R}^{m+k}\right)\longrightarrow{B}
\] 
is of the form
\[
\left(A,d_A\right)\hookrightarrow\left(A\otimes\Lambda\left(x_1,\cdots,x_{l+s},\bar{e}_{m+k-1},b_1,\cdots,b_s\right),D\right)
\]
where
\[
\begin{cases}
D_{\mid A}=d_A\\
Db_i=0\\
Dx_i=\phi\left((B\mu)^{\ast}\left(c_i\right)\right)-(B\nu)^\ast\left(c_i\right).
\end{cases}
\]
By \cite[Theorem 7.1]{cc} $(B\mu)^{\ast}\left(b_i\right)=p_i$ and $(B\nu)^\ast\left(c_i\right)=b_i$. On the other hand 
 the rational model of $f $ is completely determined by the Pontryagin classes of $\xi$ denoted $p_i\left(\xi\right)$. That is
  \[
  \phi:\left(\Lambda{\left(p_1,\cdots,p_{l}\right)},0\right)\longrightarrow \left(A,d_A\right)
  \]
   defined by
  \[
  \phi\left(p_i\right)= p_i\left(\xi\right).
  \]
  Since $\xi$ if of rank $m$, then $p_i\left(\xi\right)=0$ for $i>\frac{m}{4}$. Consequently, we have  
  
\[
\begin{cases}
D_{\mid A}=d_A\\
Db_i=0\\
Dx_i=p_i\left(\xi\right)-b_i.
\end{cases}
\]

 Now let 

 \[
 \left( A,d_A\right)\longrightarrow \left({A\otimes\Lambda{\left(x_{s+1},\cdots,x_{l+s},\bar{e}_{m+k-1}\right)}}, D\right)
 \]
\[ 
\begin{cases} 
 Da=d_Aa \text{ if } a\in A\\
 Dx_i =p_i\left(\xi\right) \text{ if } i\leq \frac{m}{4},
 Dx_i=0 \text{ if } i>\frac{m}{4}\\
 D\bar{e}_{m+k-1}=0.
 \end{cases}
 \]
 we want to prove that it's the model of 
 \[
\left(A,d_A\right)\hookrightarrow\left(A\otimes\Lambda\left(x_1,\cdots,x_{l+s},\bar{e}_{m+k-1},b_1,\cdots,b_s\right),D\right)
\]
where
\[
\begin{cases}
D_{\mid A}=d_A\\
Db_i=0\\
Dx_i=p_i\left(\xi\right)-b_i.
\end{cases}
\]
for this we define
\[
\Phi :\left(A\otimes\Lambda\left(x_1,\cdots,x_{l+s},\bar{e}_{m+k-1},b_1,\cdots,b_s\right),D\right)\longrightarrow\left({A\otimes\Lambda{\left(x_{s+1},\cdots,x_{l+s},\bar{e}_{m+k-1}\right)}}, D\right)
\]
by
\[
\begin{cases}
\Phi_{\mid A}=id\\
\Phi\left(x_i\right)= 0 \text{ for } 1\leq i\leq s\\
\Phi\left(x_i\right)=x_i \text{ for } i>s\\
\Phi\left(b_i\right)=0\\
\Phi\left(p_i\right)=p_i\left(\xi\right)
\end{cases}
\]
 
it's clear that $\Phi$ commutes with the differentials and define a quasi-isomorphism. This end the proof.
\end{proof} 
\begin{rmq}
 \label{stiefel}
   From this construction we deduce the minimal model of the Stiefel manifold of m frame in $\mathbb{R}^{m+k}$. Thus we have
   \begin{itemize}
   \item {if $k=2s+1$} we have
   \[
   V_m\left(\mathbb{R}^{m+k}\right)\simeq_{\mathbb{Q}}
   \begin{cases}
   \left(\Lambda\left(x_{s+1},\cdots,x_{l+s},\bar{e}_{m+k-1}\right),0\right) \text{ if } m=2l+1,\\
   \left(\Lambda\left(x_{s+1},\cdots,x_{l+s}\right),0\right) \text{ if } m=2l   
   \end{cases}
   \]
   \item {if $m=2l+1$ and $k=2s$} we have
    \[
   V_m\left(\mathbb{R}^{m+k}\right)\simeq_{\mathbb{Q}}
   \begin{cases}
    \left(\Lambda\left(x_s,\cdots,x_{l+s},e_k\right),d\right) \text{ where }\\
    dx_s=e_k^2, de_k=0\\
    dx_i=0 \text{ if } i>s
    \end{cases}
   \]
  
   \item {if $m=2l$ and $k=2s$} we have
    \[
   V_m\left(\mathbb{R}^{m+k}\right)\simeq_{\mathbb{Q}}
   \begin{cases}
    \left(\Lambda\left(x_s,\cdots,x_{l+s-1},\bar{e}_{m+k-1},e_k\right),d\right) \text{ where}\\
    dx_i=0 \text{ if } i>s\\
    dx_s=e_k^2, \text{ } 
    d\bar{e}_{m+k-1}=de_k=0
        \end{cases}
   \]
   
      \end{itemize}
   \end{rmq}

 From this proposition we deduce the following result
  \begin{coro}
 \label{trivial}
 If $p_i\left(\xi\right)=0,  \forall i\geq s$ the Framed bundle associated to $\xi$ with fiber $V_m\left(\mathbb{R}^{m+k}\right)$, $\Fr_m\left(\tau_M\right)$, is rationally trivial.
 \end{coro}

\section{Model of space of immersions }
In this section we give the proofs of Theorem~\ref{casod}, Theorem~\ref{caspair} and Corollary~\ref{betti}. We start with a review of Smale-Hirsch theorem which identifies the space of immersions and the space of section.

\subsection{ Associated framed bundle and space of immersions}
In this part we describe the Smale-Hirsch theorem and give its connection with the Stiefel bundle.\\
Let  $M$ and $N$ be smooth manifolds of dimension $m$ and $m+k$ respectively. We suppose that $k>0$ and that $M$ and $N$ are simply connected. We denote by $TM$ and $TN$ the tangent of $M$ and $N$ respectively. We set $T_m=P_mTM$ and $T_mN=P_mTN$, $P_mTM$ and $P_mTN$ are the space of $m$- frame in $TM$ and $TN$ defined in section~\ref{stiefelbundle}. Then $\Fr_m\left(\tau_M\right)=\left(T_mM\times_{SO\left(m\right)}T_mN,p,M\right)$ denotes the framed bundle associated to $TM$ with fiber $T_mN$. Let $\Gamma\left(\Fr_m\left(\tau_M\right)\right)$ be the space of sections of $\left(T_mM\times_{SO\left(m\right)}T_mN,p,M\right)$. That is 
\[
\Gamma\left(\Fr_m\left(\tau_M\right)\right)=\{ s:M\longrightarrow T_mM\times_{SO\left(m\right)}T_mN\mid ps=id\}.
\]  
We will construct a map from $Imm\left(M,N\right)$ to $\Gamma\left(p\right)$. For this, let $f\in Imm\left(M,N\right)$ it induces an $SO\left(m\right)$ equivariant map from $f^\ast:T_mM\longrightarrow T_mN$ the map which send the  $m$ frames of $T_mM$ into $m$- frames of $T_mN$. Then the map

\begin{eqnarray*}
T_mM&\longrightarrow&T_mM\times T_mN\\
X&\mapsto&\left(X,f^\ast\left(X\right)\right)
\end{eqnarray*}
is also $SO\left(m\right)$-equivariant. Passing to $SO\left(m\right)-$ orbits yields 
\[
s_f: M\longrightarrow T_mM\times_{SO\left(m\right)}T_mN
\]
which is in fact a section of $\Fr_m\left(\tau_M\right)=\left(T_mM\times_{SO\left(m\right)}T_mN,p,M\right)$. 
Morris Hirsch proved in \cite{hirsh} the following result.
\begin{thm}
\label{hirsch}
The map 
\begin{eqnarray*}
Imm\left(M,N\right)&\longrightarrow &\Gamma\left(\Fr_m\left(\tau_M\right)\right)\\
f&\mapsto&s_f
\end{eqnarray*}
is a homotopy equivalence.

\end{thm}
 \subsection{Rational model of the space of immersions }
 
  In this part we give the proofs of Theorem~\ref{casod}, and Theorem~\ref{caspair}.

\begin{thm}
\label{casod}
Let $M$ be simply connected manifold of dimension $m$ and $k=2s+1$ with $s\neq 0$. Suppose that $p_i\left(\tau_M\right)=0, \forall i\geq s+1$. Then, each component $Imm\left(M ,\mathbb{R}^{m+k}\right)$
has the rational homotopy type of the product of Eilenberg Maclane spaces, which depends only on the rational Betti numbers of $M$ and $k$. More precisely we have:
\[
Imm\left(M ,\mathbb{R}^{m+k},f\right)\simeq_{\mathbb{Q}}
\begin{cases}
\prod_{s\leq i\leq l+s}\prod_{1\leq q\leq 4i-1}K\left(H^{4i-1-q}\left(M,\mathbb{Q}\right),q\right) \text{ si } m=2l\\
\prod_{s\leq i\leq l+s}\prod_{1\leq q\leq 4i-1}K\left(H^{4i-1-q}\left(M,\mathbb{Q}\right),q\right)\times {K}\text{ si } m=2l+1
\end{cases}
\]
where
\[
K={\prod_{1\leq j\leq m+k-1}K\left(H^{m+k-1-j}\left(M,\mathbb{Q}\right),j\right)}
\]
and
$Imm\left(M ,\mathbb{R}^{m+k},f\right)$ is the component of $Imm\left(M ,\mathbb{R}^{m+k}\right)$ contains $f$.
\end{thm}

\begin{proof}
 Let $f:M\looparrowright \mathbb{R}^{m+k}$  be an immersion and
  \[
  s_f: M\longrightarrow{T_mM\times_{SO\left(m\right)}V_m\left(\mathbb{R}^{m+k}\right)}
  \]
  the section of 
 $ \Fr_m\left(\tau_M\right)$ induced by $f$. By Theorem~\ref{hirsch} we have:
\[ 
Imm\left(M, \mathbb{R}^{m+k},f\right)\simeq{\Gamma\left(\Fr_m\left(\tau_M\right),s_f\right)}
\]
where $\Gamma\left(p,s_f\right)$ is the space of sections  of $p:T_mM\times_{SO\left(m\right)}V_m\left(\mathbb{R}^{m+k}\right)\longrightarrow M$ homotope
to $s_f$.\\
Denote 
\[
\Fr_m\left(\tau_m\right)_{\left(\mathbb{Q}\right)}=\left(\left(T_mM\times_{SO\left(m\right)}V_m\left(\mathbb{R}^{m+k}\right)\right)_{\left(\mathbb{Q}\right)}, p_{\left(\mathbb{Q}\right)}, M\right)
\] 
the fiberwise rationalization  of $p:T_mM\times_{SO\left(m\right)}V_m\left(\mathbb{R}^{m+k}\right)\longrightarrow M$ and 
\[
(s_f)_{(\mathbb{Q})} :M\longrightarrow \left(T_mM\times_{SO\left(m\right)}V_m\left(\mathbb{R}^{m+k}\right)\right)_{\left(\mathbb{Q}\right)}
\]
the induced section.\\
Since  $V_m\left(\mathbb{R}^{m+k}\right) $ is simply connected, because $k>1$, and  $M$ is simply connected and of finite type, the bundle $\left(T_mM\times_{SO\left(m\right)}T_m\left(\mathbb{R}^{m+k}\right),p,M\right)$ is nilpotent. Then by \cite[theorem 5.3]{Moller}  $\Gamma\left(\Fr_m\left(\tau_M\right),s_f\right)$ is nilpotent and we have the following rational equivalence 
\[
\Gamma\left(\Fr_m\left(\tau_M\right),s_f\right)\simeq_{\mathbb{Q}}\Gamma\left(\Fr_m\left(\tau_M\right)_{(\mathbb{Q})},(s_f)_{(\mathbb{Q})}\right).
\]
Since
 \[
 p_i\left(\tau_M\right)=0\quad \forall\quad i\geq s+1
 \]
  the Corollary~\ref{trivial} imply that $\Fr_m\left(\tau_M\right)$ is rationally trivial. And by Proposition~\ref{lupton}, $\Fr_m\left(\tau_M\right)_{\left(\mathbb{Q}\right)}$ is rationally trivial , and we have 
\[
 \Gamma\left(\Fr_m\left(\tau_M\right)_{(\mathbb{Q})},(s_f)_{(\mathbb{Q})}\right)\simeq Map\left(M ,({V_m}\left(\mathbb{R}^{m+k}\right))_{\mathbb{Q}},(\tilde{s_f})_{(\mathbb{Q})}\right).
 \]
 Since k is odd, by corollary~\ref{stiefel}, the Stiefel manifold ${V_m}\left(\mathbb{R}^{m+k}\right)$ has the rational homotopy type of product of Eilenberg Maclane spaces
  \[
{V_m}\left(\mathbb{R}^{m+k}\right)\simeq_{\mathbb{Q}}
\begin{cases}
\prod_{s\leq i\leq l+s}\emph{K}\left(\mathbb{Q},4i-1\right) \text{ if } m=2l\\
\prod_{s\leq i\leq l+s}\emph{K}\left(\mathbb{Q},4i-1\right)\times \emph{K}\left(\mathbb{Q},m+k-1\right) \text{ if } m=2l+1
\end{cases}
\]
thus
\[
Map\left(M ,{V_m}\left(\mathbb{R}^{m+k}\right),\tilde{s_f}\right)\simeq_{\mathbb{Q}}
\begin{cases}
\prod_{s\leq i\leq l+s}Map\left(M,K\left(\mathbb{Q},4i-1\right) \tilde{s}_{i}\right) \text{ if } m=2l\\
\prod_{s\leq i\leq l+s}Map\left(M,K\left(\mathbb{Q},4i-1,\tilde{s}_{i}\right)\right)\times K \text{ if } m=2l+1
\end{cases}
\]
where $\tilde{s}_{i}: M\longrightarrow K\left(\mathbb{Q},4i-1\right)$ are conponents of $(\tilde{s_f})_{(\mathbb{Q})}$.\\
By \cite{thom} each component of $Map\left(M,K\left(\mathbb{Q},4i-1\right) \right)$ has the rational homotopy type of the product of Eilenberg
Maclane spaces which depend only with rational cohomology of $M$. That is
\[
Map\left(M,K\left(\mathbb{Q},4i-1\right),s_i\right)\simeq\prod_{1\leq q\leq 4i-1}K\left(H^{4i-1-q}\left(M,\mathbb{Q}\right),q\right)
\]
therefore
\[
Map\left(M ,{V_m}\left(\mathbb{R}^{m+k}\right),\tilde{s}_f\right)\simeq_{\mathbb{Q}}
\begin{cases}
\prod_{s\leq i\leq l+s}\prod_{1\leq q\leq 4i-1}K\left(H^{4i-1-q}\left(M,\mathbb{Q}\right),q\right) \text{ if } m=2l\\
\prod_{s\leq i\leq l+s}\prod_{1\leq q\leq 4i-1}K\left(H^{4i-1-q}\left(M,\mathbb{Q}\right),q\right)\times {K}\text{ if } m=2l+1
\end{cases}
\]
where
\[
K={\prod_{1\leq j\leq m+k-1}K\left(H^{m+k-1-j}\left(M,\mathbb{Q}\right),j\right)}.
\]

\end{proof}
\begin{thm}
\label{caspair} 
Let $M$ be a simply connected $m$-manifold of finite type and $k$ and an even integer. Then:
\begin{itemize}
\item
If $k\geq m+1$,  $Imm\left(M,\mathbb{R}^{m+k}\right)$ is connected and its rational homotopy type is of the form
 \[
 Imm\left(M, \mathbb{R}^{m+k}\right)\simeq_{\mathbb{Q}}Map\left(M ,S^k\right)\times K,
 \]
 where
 $Map\left(M ,S^k\right)$ is the mapping space of maps from $M$ into $S^k$ and $K$ is a product of Eilenberg-Maclane spaces which depend only on the Betti numbers of $M$.
 \item
 If $2\leq k\leq m$,  Suppose that $p_i\left(\tau_M\right)=0, \forall i\geq s$ then, the connected component of $f$ in $Imm\left(M,\mathbb{R}^{m+k}\right)$ has the rational homotopy type of a product
 \[
  Imm\left(M, \mathbb{R}^{m+k},f\right)\simeq_{\mathbb{Q}}Map\left(M ,S^k,\tilde{f}\right)\times K,
 \]
where
 $ Imm\left(M, \mathbb{R}^{m+k},f\right)$ is the component of  $Imm\left(M, \mathbb{R}^{m+k}\right)$ which contains $f$,\\
$ K$ is the product of Eilenberg-Maclane spaces which depends only on the Betti numbers of $M$,\\
 $Map\left(M ,S^k,\tilde{f}\right)$  is the path component of the  mapping space of $M$ in $S^k$ containing the map $\tilde{f}$ induced by the immersion $f$.
 \end{itemize}
\end{thm}
\begin{proof}

By Theorem~\ref{hirsch}, we have a homotopy equivalence between the space of immersions  of $M$ in $\mathbb{R}^{m+k}$ and the space of sections of $\Fr_m\left(\tau_M\right)$
\[
Imm\left(M, \mathbb{R}^{m+k}\right)\simeq{\Gamma\left(\Fr_m\left(\tau_M\right)\right)}.
\]
We distinguish two cases:
\begin{itemize}
\item
If $k\geq m+1$ the connectivity of $V_m\left(\mathbb{R}^{m+k}\right)$ is greater than the dimension of $M$, therefore the space of sections  of
$\Fr_m\left(\tau_M\right)$ is connected and subsequently $Imm\left(M,\mathbb{R}^{m+k}\right)$  is connected.\\
In the other hand,  for all $i\geq s, p_i\left(\tau_M\right)=0$ since $s>l$. By Corollary~\ref{trivial}, $\Fr_m\left(\tau_M\right)$ is
rationally trivial and we have
\[
 \Gamma\left( \Theta_{\mathbb{R}^{m+k}}\left(\tau_M\right)_{\mathbb{Q}}\right)\simeq Map\left(M ,{V_m}\left(\mathbb{R}^{m+k}\right)\right).
 \]
 By Remark~\ref{stiefel}
   \[
{V_m}\left(\mathbb{R}^{m+k}\right)\simeq_{\mathbb{Q}}
\begin{cases}
\prod_{s+1\leq i\leq l+s-1}\emph{K}\left(\mathbb{Q},4i-1\right)\times S^k \text{ if } m=2l+1\\
\prod_{s+1\leq i\leq l+s-1}\emph{K}\left(\mathbb{Q},4i-1\right)\times \emph{K}\left(\mathbb{Q},m+k-1\right)\times S^k \text{ if } m=2l
\end{cases}
\]
therefore,
\[
Map\left(M ,{V_m}\left(\mathbb{R}^{m+k}\right)\right)\simeq_{\mathbb{Q}}
\begin{cases}
\prod_{s+1\leq i\leq l+s}Map\left(M,K\left(\mathbb{Q},4i-1\right)\right)\times {H}\text{ if } m=2l+1\\
\prod_{s+1\leq i\leq l+s}Map\left(M,K\left(\mathbb{Q},4i-1\right)\right)\times{A} \times {H} \text{ if } m=2l
\end{cases}
\]
where
\[
H=Map\left(M,S^k\right)
\] 
and
\[
A=Map\left(M,K\left(\mathbb{Q},m+k-1\right)\right)
\]
finally , by \cite{ thom}, we have
\[
Imm\left(M ,\mathbb{R}^{m+k}\right)\simeq_{\mathbb{Q}}
\begin{cases}
\prod_{s+1\leq i\leq l+s}\prod_{1\leq q\leq 4i-1}K\left(H^{4i-1-q}\left(M,\mathbb{Q}\right),q\right)\times {H} \text{ si } m=2l+1\\
\prod_{s+1\leq i\leq l+s}\prod_{1\leq q\leq 4i-1}K\left(H^{4i-1-q}\left(M,\mathbb{Q}\right),q\right)\times {H} \times {K}\text{ si } m=2l
\end{cases}
\]
where
\[
K={\prod_{1\leq j\leq m+k-1}K\left(H^{m+k-1-j}\left(M,\mathbb{Q}\right),j\right)}.
\]
\item{ If $2\leq k\leq m$}, in this cas $Imm\left(M, \mathbb{R}^{m+k}\right)$ is not necessarily  connected. Since $p_i\left(\tau_M\right)=0, \forall i\geq s$
  by the same arguments as in the proof of Theorem~\ref{casod} for  $f\in Imm\left(M, \mathbb{R}^{m+k}\right)$we obtain:
\[
Imm\left(M ,\mathbb{R}^{m+k},f\right)\simeq_{\mathbb{Q}}
\begin{cases}
\prod_{s+1\leq i\leq l+s}\prod_{1\leq q\leq 4i-1}K\left(H^{4i-1-q}\left(M,\mathbb{Q}\right),q\right)\times  {G}\text{ if } m=2l+1\\
\prod_{s+1\leq i\leq l+s}\prod_{1\leq q\leq 4i-1}K\left(H^{4i-1-q}\left(M,\mathbb{Q}\right),q\right)\times {G}\times {K}\text{ if } m=2l
\end{cases}
\]
where
\[
G=Map\left(M,S^k,f_k\right)
\]
and
\[
K={\prod_{1\leq j\leq m+k-1}K\left(H^{m+k-1-j}\left(M,\mathbb{Q}\right),j\right)}.
\]

\end{itemize}
\end{proof}
\begin{coro}
\label{idem}
If $H^k\left(M,\mathbb{Q}\right)=0$, then all the path connected of $Imm\left(M,\mathbb{R}^{m+k}\right)$ have the same rational homotopy type, and we can take $\tilde{f}$ to be the constant map.
\end{coro}
\begin{proof}
It suffices to prove that for each $f: M\longrightarrow S^k$ we have
\[
Map(M,S^k,f)\simeq_{\mathbb{Q}} Map(M,S^k,0).
\]
Let $(\Lambda(x,y),d)$ be the model of $S^k$ and $(A,d)$ a finite dimensional model of $M$. Let 
\[
\phi: (\Lambda(x,y),d)\longrightarrow (A,d_A)
\]
be the model of $f$. Since $H^k(A,d)=0$, then $\phi$ is homotope to some map which send $x$ to $0$. Replace $\phi$ by this map. In this case $\phi(y)=a$ is a cocycle.\\
Now  we consider 
\[
A\otimes \Lambda(x,y)\overset{\sigma=id\otimes\phi}\longrightarrow A
\]
and proceed to a change of variable in order to have
\[
\sigma(x)=\sigma(y)=0.
\]
Pose $y'=y-\phi(y)=y-a$ and $dy'=dy$. \\
We can suppose with the same differential that $\sigma(x)=\sigma(y)=0$.
All the following steps in the model of  $Map(M,S^k,f)$ depend only on $\Lambda(x,y),A$ and $\sigma$, the result is the same for each $f$. Henceforth we
have 
\[
Map(M,S^k,f)\simeq_{\mathbb{Q}} Map(M,S^k,0).
\]
 Thenceforward all the components of $Imm\left(M ,\mathbb{R}^{m+k},f\right)$ have the same rational homotopy type. 
\end{proof}
From this result we deduce 
\begin{coro} 
\label{betti}
If $M$ is a simply connected manifold of dimension $m$ and $k\geq2$ an integer. Then the rational Betti numbers of each component of the space of immersions of $M$ in $\mathbb{R}^{m+k}$ have polynomial growth.
\end{coro}
\begin{rmq}
\label{alp}
In \cite{alp} Arone, Lambrechts and Pryor prove that, if $\chi\left(M\right)\leq -2$ and $k\geq {m+1}$ the rational Betti numbers of space of smooth embeddings of $M$ in $\mathbb{R}^{m+k} $ modulo immersions, $\overline{Emb}\left(M ,\mathbb{R}^{m+k}\right)$, have exponential growth. By corollary~\ref{betti} we deduce that the Betti numbers of smooth embedding $Emb\left(M ,\mathbb{R}^{m+k}\right)$ have exponential growth.
\end{rmq}

\bibliographystyle{plain}
\bibliography{mabiblio-2}

\textsf{Universit\'e catholique de Louvain, Chemin du Cyclotron 2, B-1348 Louvain-la-Neuve, Belgique\\
Institut de Recherche en Math\'ematique et Physique\\}
\textit{E-mail address: abdoul.yacouba@uclouvain.be}
\end{document}